\documentclass[thmsa,onecolumn,10pt]{article}
\usepackage{amssymb}
\usepackage{amsfonts}
\usepackage{graphicx}
\usepackage{amsmath}

\newtheorem{theorem}{Theorem}

\newtheorem{corollary}[theorem]{Corollary}

\newtheorem{proposition}[theorem]{Proposition}

\newenvironment{proof}[1][Proof]{\textbf{#1.} }{\ \rule{0.5em}{0.5em}}
\begin{document}

\title{\textsc{Burkholder}'\textsc{s submartingales from a stochastic
calculus perspective}}
\author{Giovanni PECCATI\thanks{%
Laboratoire de Statistique Th\'{e}orique et Appliqu\'{e}e, Universit\'{e}
Paris VI, France. E-mail: \texttt{giovanni.peccati@gmail.com}}, and Marc YOR%
\thanks{%
Laboratoire de Probabilit\'{e}s et Mod\`{e}les Al\'{e}atoires, Universit\'{e}%
s Paris VI and Paris VII, France and Institut Universitaire de France.}}
\date{May 24, 2007}
\maketitle

\begin{abstract}
We provide a simple proof, as well as several generalizations, of a recent
result by Davis and Suh, characterizing a class of continuous submartingales
and supermartingales that can be expressed in terms of a squared Brownian
motion and of some appropriate powers of its maximum. Our techniques involve
elementary stochastic calculus, as well as the Doob-Meyer decomposition of
continuous submartingales. These results can be used to obtain an explicit
expression of the constants appearing in the Burkholder-Davis-Gundy
inequalities. A connection with some \textsl{balayage formulae }is also
established.


\textbf{AMS\ 2000 classification: \ }60G15, 60G44
\end{abstract}

\section{Introduction}

Let $W=\left\{ W_{t}:t\geq 0\right\} $ be a standard Brownian motion
initialized at zero, set $W_{t}^{\ast }=\max_{s\leq t}\left\vert
W_{s}\right\vert $ and write $\mathcal{F}_{t}^{W}=\sigma \left\{ W_{u}:u\leq
t\right\} $, $t\geq 0$. In \cite{DS}, Davis and Suh proved the following
result.

\begin{theorem}[{\protect\cite[Th. 1.1]{DS}}]
\label{Th : DavisSuh}For every $p>0$ and every $c\in \mathbb{R}$, set
\begin{eqnarray}
Y_{t} &=&Y_{t}\left( c,p\right) =\left( W_{t}^{\ast }\right) ^{p-2}\left[
W_{t}^{2}-t\right] +c\left( W_{t}^{\ast }\right) ^{p}\text{, \ \ }t>0\text{,}
\label{Y} \\
Y_{0}\left( c,p\right) &=&Y_{0}=0.  \notag
\end{eqnarray}

\begin{enumerate}
\item For every $p\in (0,2]$, the process $Y_{t}$ is a $\mathcal{F}_{t}^{W}$%
-submartingale if, and only if, $c\geq \frac{2-p}{p}.$

\item For every $p\in \lbrack 2,+\infty )$, the process $Y_{t}$ is a $%
\mathcal{F}_{t}^{W}$-supermartingale if, and only if, $c\leq \frac{2-p}{p}.$
\end{enumerate}
\end{theorem}

\bigskip

As pointed out in \cite[p. 314]{DS} and in Section \ref{S : BDG} below, part
1 of Theorem \ref{Th : DavisSuh} can be used to derive explicit expressions
of the constants appearing in the Burkholder-Davis-Gundy (BDG) inequalities
(see \cite{Burk73}, or \cite[Ch. IV, \S 4]{RY}). The proof of Theorem \ref%
{Th : DavisSuh} given in \cite{DS} uses several delicate estimates related
to a class of Brownian hitting times: such an approach can be seen as a
ramification of the discrete-time techniques developed in \cite{Burkhj}. In
particular, in \cite{DS} it is observed that the submartingale (or
supermartingale) characterization of $Y_{t}\left( c,p\right) $ basically
relies on the properties of the random subset of $[0,+\infty )$ composed of
the instants $t$ where $\left\vert W_{t}\right\vert =W_{t}^{\ast }$. The aim
of this note is to bring this last connection into further light, by
providing an elementary proof of Theorem \ref{Th : DavisSuh}, based on a
direct application of It\^{o} formula and on an appropriate version of the
Doob-Meyer decomposition of submartingales. We will see that our techniques
lead naturally to some substantial generalizations (see Theorem \ref{T : DS}
below).

\bigskip

The rest of the paper is organized as follows. In Section \ref{S : Gen} we
state and prove a general result involving a class of stochastic processes
that are functions of a positive submartingale and of a monotone
transformation of its maximum. In Section \ref{S : Proof} we focus once
again on the Brownian setting, and establish a generalization of Theorem \ref%
{Th : DavisSuh}. Section \ref{S : BDG} deals with an application of the
previous results to (strong) BDG inequalities. Finally, in Section \ref{S :
Bal} we provide an explicit connection with some classic \textsl{balayage
formulae} for continuous-time semimartingales (see e.g. \cite{MY}).

All the objects appearing in the subsequent sections are defined on a common
probability space $\left( \Omega ,\mathfrak{A},\mathbb{P}\right) $.

\section{A general result \label{S : Gen}}

Throughout this section, $\mathcal{F}=\left\{ \mathcal{F}_{t}:t\geq
0\right\} $ stands for a filtration satisfying the usual conditions. We will
write $X=\left\{ X_{t}:t\geq 0\right\} $ to indicate a \textsl{continuous\ }$%
\mathcal{F}_{t}$-\textsl{submartingale }issued from zero and such that $%
\mathbb{P}\left\{ X_{t}\geq 0\text{, \ }\forall t\right\} =1$. We will
suppose that the Doob-Meyer decomposition of $X$ (see for instance \cite[Th.
1.4.14]{KS}) is of the type $X_{t}=M_{t}+A_{t}$, $t\geq 0$, where $M$ is a
\textsl{square-integrable} continuous $\mathcal{F}_{t}$-martingale issued
from zero, and $A$ is an increasing (integrable) natural process. We assume
that $A_{0}=M_{0}=0$; the symbol $\left\langle M\right\rangle =\left\{
\left\langle M\right\rangle _{t}:t\geq 0\right\} $ stands for the quadratic
variation of $M$. We note $X_{t}^{\ast }=\max_{s\leq t}X_{s}$, and we also
suppose that $\mathbb{P}\left\{ X_{t}^{\ast }>0\right\} =1$ for every $t>0$.
The following result is a an extension of Theorem \ref{Th : DavisSuh}.

\begin{theorem}
\label{Th : general}Fix $\varepsilon >0$.

\begin{enumerate}
\item Suppose that the function $\phi :(0,+\infty )\mapsto \mathbb{R}$ is of
class $C^{1}$, non-increasing, and such that
\begin{equation}
\mathbb{E[}\int_{\varepsilon }^{T}\phi \left( X_{s}^{\ast }\right)
^{2}d\left\langle M\right\rangle _{s}]<+\infty \text{,}  \label{int}
\end{equation}%
for every $T>\varepsilon $. For every $x\geq z>0$, we set
\begin{equation}
\Phi \left( x,z\right) =-\int_{z}^{x}y\phi ^{\prime }\left( y\right) dy;
\label{PHI}
\end{equation}%
then, for every $\alpha \geq 1$ the process
\begin{equation}
Z_{\varepsilon }\left( \phi ,\alpha ;t\right) =\phi \left( X_{t}^{\ast
}\right) \left( X_{t}-A_{t}\right) +\alpha \Phi \left( X_{t}^{\ast
},X_{\varepsilon }^{\ast }\right) \text{, \ }t\geq \varepsilon \text{,}
\label{subsup}
\end{equation}%
is a $\mathcal{F}_{t}$-submartingale on $[\varepsilon ,+\infty )$.

\item Suppose that the function $\phi :(0,+\infty )\mapsto \mathbb{R}$ is of
class $C^{1}$, non-decreasing and such that (\ref{int}) holds for every $%
T>\varepsilon .$ Define $\Phi \left( \cdot ,\cdot \right) $ according to (%
\ref{PHI}), and $Z_{\varepsilon }\left( \phi ,\alpha ;t\right) $ according
to (\ref{subsup}). Then, for every $\alpha \geq 1$ the process $%
Z_{\varepsilon }\left( \phi ,\alpha ;t\right) $ is a $\mathcal{F}_{t}$%
-supermartingale on $[\varepsilon ,+\infty )$.
\end{enumerate}
\end{theorem}

\bigskip

\textbf{Remarks. }(i) Note that the function $\phi \left( y\right) $ (and $%
\phi ^{\prime }\left( y\right) $) need not be defined at $y=0$.

(ii) In Section \ref{S : Proof}, where we will focus on the Brownian
setting, we will exhibit specific examples where the condition $\alpha \geq
1 $ is necessary and sufficient to have that the process $Z_{\varepsilon
}\left( \alpha ,\phi ;t\right) $ is a submartingale (when $\phi $ is
non-increasing) or a supermartingale (when $\phi $ is non-decreasing).

\bigskip

\textbf{Proof of Theorem \ref{Th : general}. }(\textsl{Proof of Point 1.})
Observe first that, since $M_{t}=X_{t}-A_{t}$ is a continuous martingale, $%
X^{\ast }$ is non-decreasing and $\phi $ is differentiable, then a standard
application of It\^{o} formula gives that
\begin{eqnarray}
\phi \left( X_{t}^{\ast }\right) \left( X_{t}-A_{t}\right) -\phi \left(
X_{\varepsilon }^{\ast }\right) \left( X_{\varepsilon }-A_{\varepsilon
}\right) &=&\phi \left( X_{t}^{\ast }\right) M_{t}-\phi \left(
X_{\varepsilon }^{\ast }\right) M_{\varepsilon }  \notag \\
&=&\int_{\varepsilon }^{t}\phi (X_{s}^{\ast })dM_{s}+\int_{\varepsilon
}^{t}\left( X_{s}-A_{s}\right) \phi ^{\prime }\left( X_{s}^{\ast }\right)
dX_{s}^{\ast }.  \label{uyu}
\end{eqnarray}%
The assumptions in the statement imply that the application $\widetilde{M}%
_{\varepsilon ,t}:=\int_{\varepsilon }^{t}\phi (X_{s}^{\ast })dM_{s}$ is a
continuous square integrable $\mathcal{F}_{t}$-martingale on $[\varepsilon
,+\infty )$. Moreover, the continuity of $X$ implies that the support of the
random measure $dX_{t}^{\ast }$ (on $[0,+\infty )$) is contained in the
(random) set $\left\{ t\geq 0:X_{t}=X_{t}^{\ast }\right\} $, thus yielding
that
\begin{eqnarray*}
\int_{\varepsilon }^{t}\left( X_{s}-A_{s}\right) \phi ^{\prime }\left(
X_{s}^{\ast }\right) dX_{s}^{\ast } &=&\int_{\varepsilon }^{t}\left(
X_{s}^{\ast }-A_{s}\right) \phi ^{\prime }\left( X_{s}^{\ast }\right)
dX_{s}^{\ast } \\
&=&-\int_{\varepsilon }^{t}A_{s}\phi ^{\prime }\left( X_{s}^{\ast }\right)
dX_{s}^{\ast }-\Phi \left( X_{t}^{\ast },X_{\varepsilon }^{\ast }\right) ,
\end{eqnarray*}%
where $\Phi $ is defined in (\ref{PHI}). As a consequence,
\begin{equation}
Z_{\varepsilon }\left( \phi ,\alpha ;t\right) =\widetilde{M}_{\varepsilon
,t}+\int_{\varepsilon }^{t}(-A_{s}\phi ^{\prime }\left( X_{s}^{\ast }\right)
)dX_{s}^{\ast }+\left( \alpha -1\right) \Phi \left( X_{t}^{\ast
},X_{\varepsilon }^{\ast }\right) \text{.}  \label{a}
\end{equation}%
Now observe that the application $t\mapsto \Phi \left( X_{t}^{\ast
},X_{\varepsilon }^{\ast }\right) $ is non-decreasing (a.s.-$\mathbb{P)}$,
and also that, by assumption, $-A_{s}\phi ^{\prime }\left( X_{s}^{\ast
}\right) \geq 0$ for every $s>0$. This entails immediately that $%
Z_{\varepsilon }\left( \phi ,\alpha ;t\right) $ is a $\mathcal{F}_{t}$%
-submartingale for every $\alpha \geq 1$.

(\textsl{Proof of Point 2.}) By using exactly the same line of reasoning as
in the proof of Point 1., we obtain that
\begin{equation}
Z_{\varepsilon }\left( \phi ,\alpha ;t\right) =\int_{\varepsilon }^{t}\phi
(X_{s}^{\ast })dM_{s}+\int_{\varepsilon }^{t}(-A_{s}\phi ^{\prime }\left(
X_{s}^{\ast }\right) )dX_{s}^{\ast }+\left( \alpha -1\right) \Phi \left(
X_{t}^{\ast },X_{\varepsilon }^{\ast }\right) .  \label{aa}
\end{equation}%
Since (\ref{int}) is in order, we deduce that $t\mapsto \int_{\varepsilon
}^{t}\phi (X_{s}^{\ast })dM_{s}$ is a continuous (square-integrable) $%
\mathcal{F}_{t}$-martingale on $[\varepsilon ,+\infty )$. Moreover, $%
-A_{s}\phi ^{\prime }\left( X_{s}^{\ast }\right) \leq 0$ for every $s>0$,
and we also have that $t\mapsto \Phi \left( X_{t}^{\ast },X_{\varepsilon
}^{\ast }\right) $ is a.s. decreasing. This implies that $Z_{\varepsilon
}\left( \phi ,\alpha ;t\right) $ is a $\mathcal{F}_{t}$-supermartingale for
every $\alpha \geq 1$. \ \ $\blacksquare $

\bigskip

The next result allows to characterize the nature of the process $Z$
appearing in (\ref{subsup}) on the whole positive axis. Its proof can be
immediately deduced from formulae (\ref{a}) (for Part 1) and (\ref{aa}) (for
Part 2).

\begin{proposition}
\label{P : epsilon}Let the assumptions and notation of this section prevail.

\begin{enumerate}
\item Consider a decreasing function $\phi :(0,+\infty )\mapsto \mathbb{R}$
verifying the assumptions of Part 1 of Theorem \ref{Th : general} and such
that
\begin{equation}
\Phi \left( x,0\right) :=-\int_{0}^{x}y\phi ^{\prime }\left( y\right) dy%
\text{ is finite }\forall x>0\text{.}  \label{f1}
\end{equation}%
Assume moreover that
\begin{equation}
\mathbb{E[}\int_{0}^{T}\phi \left( X_{s}^{\ast }\right) ^{2}d\left\langle
M\right\rangle _{s}]<+\infty \text{,}  \label{f2}
\end{equation}%
and also
\begin{eqnarray}
&&\phi \left( X_{\varepsilon }^{\ast }\right) M_{\varepsilon }=\phi \left(
X_{\varepsilon }^{\ast }\right) \left( X_{\varepsilon }-A_{\varepsilon
}\right) \text{ converges to zero in }L^{1}\left( \mathbb{P}\right) ,\text{
as }\varepsilon \downarrow 0,  \label{f3} \\
&&\Phi \left( X_{t}^{\ast },0\right) \in L^{1}\left( \mathbb{P}\right) \text{%
.}  \label{f4}
\end{eqnarray}%
Then, for every $\alpha \geq 1$ the process
\begin{equation}
Z\left( \phi ,\alpha ;t\right) =\left\{
\begin{array}{ll}
0 & \text{for }t=0 \\
\phi \left( X_{t}^{\ast }\right) \left( X_{t}-A_{t}\right) +\alpha \Phi
\left( X_{t}^{\ast },0\right) & \text{for }t>0%
\end{array}%
\right. ,  \label{zz}
\end{equation}%
is a $\mathcal{F}_{t}$-submartingale.

\item Consider an increasing function $\phi :(0,+\infty )\mapsto \mathbb{R}$
as in Part 2 of Theorem \ref{Th : general} and such that assumptions (\ref%
{f1})--(\ref{f4}) are satisfied. Then, for every $\alpha \geq 1$ the process
$Z\left( \phi ,\alpha ;t\right) $ appearing in (\ref{zz}) is a $\mathcal{F}%
_{t}$-supermartingale.
\end{enumerate}
\end{proposition}

\bigskip

\textbf{Remarks. }(i) A direct application of the Cauchy-Schwarz inequality
shows that a sufficient condition to have (\ref{f3}) is the following:%
\begin{equation}
\lim_{\varepsilon \downarrow 0}\mathbb{E}\left[ \phi \left( X_{\varepsilon
}^{\ast }\right) ^{2}\right] \times \mathbb{E}\left[ M_{\varepsilon }^{2}%
\right] =\lim_{\varepsilon \downarrow 0}\mathbb{E}\left[ \phi \left(
X_{\varepsilon }^{\ast }\right) ^{2}\right] \times \mathbb{E}\left[
\left\langle M\right\rangle _{\varepsilon }\right] =0  \label{cv0}
\end{equation}%
(observe that $\lim_{\varepsilon \downarrow 0}\mathbb{E}\left[
M_{\varepsilon }^{2}\right] =0$, since $M_{0}=0$ by assumption). In other
words, when (\ref{cv0}) is verified the quantity $\mathbb{E}\left[
M_{\varepsilon }^{2}\right] $ `takes care' of the possible explosion of $%
\varepsilon \mapsto \mathbb{E}\left[ \phi \left( X_{\varepsilon }^{\ast
}\right) ^{2}\right] $ near zero.

(ii) Let $\phi $ be non-increasing or non-decreasing on $\left( 0,+\infty
\right) $, and suppose that $\phi $ satisfies the assumptions of Theorem \ref%
{Th : general} and Proposition \ref{P : epsilon}. Then, the process $%
t\mapsto \int_{0}^{t}\phi (X_{s}^{\ast })dM_{s}$ is a continuous
square-integrable $\mathcal{F}_{t}^{W}$-martingale. Moreover, for any choice
of $\alpha \in \mathbb{R}$, the process $Z\left( \phi ,\alpha ;t\right) $, $%
t\geq 0$, defined in (\ref{zz}) is a semimartingale, with canonical
decomposition given by
\begin{equation*}
Z\left( \phi ,\alpha ;t\right) =\int_{0}^{t}\phi (X_{s}^{\ast
})dM_{s}+\int_{0}^{t}\left( (\alpha -1)X_{s}^{\ast }-A_{s}\right) \phi
^{\prime }\left( X_{s}^{\ast }\right) dX_{s}^{\ast }.
\end{equation*}

\section{A generalization of Theorem \protect\ref{Th : DavisSuh} \label{S :
Proof}}

The forthcoming Theorem \ref{T : DS} is a generalization of Theorem \ref{Th
: DavisSuh}. Recall the notation: $W$ is a standard Brownian motion issued
from zero, $W_{t}^{\ast }=\max_{s\leq t}\left\vert W_{s}\right\vert $ and $%
\mathcal{F}_{t}^{W}=\sigma \left\{ W_{u}:u\leq t\right\} $. We also set for
every $m\geq 1$, every $p>0$ and every $c\in \mathbb{R}$:
\begin{eqnarray}
J_{t} &=&J_{t}\left( m,c,p\right) =\left( W_{t}^{\ast }\right) ^{p-m}\left[
\left\vert W_{t}\right\vert ^{m}-A_{m,t}\right] +c\left( W_{t}^{\ast
}\right) ^{p}\text{, \ \ }t>0\text{,}  \label{ggei} \\
J_{0}\left( m,c,p\right) &=&J_{0}=0,  \notag
\end{eqnarray}%
where $t\mapsto A_{m,t}$ is the increasing natural process in the Doob-Meyer
decomposition of the $\mathcal{F}_{t}^{W}$-submartingale $t\mapsto
\left\vert W_{t}\right\vert ^{m}$. Of course, $J_{t}\left( 2,c,p\right)
=Y_{t}\left( c,p\right) $, as defined in (\ref{Y}).

\begin{theorem}
\label{T : DS}Under the above notation:

\begin{enumerate}
\item For every $p\in (0,m]$, the process $J_{t}$ is a $\mathcal{F}_{t}^{W}$%
-submartingale if, and only if, $c\geq \frac{m-p}{p}.$

\item For every $p\in \lbrack m,+\infty )$, the process $J_{t}$ is a $%
\mathcal{F}_{t}^{W}$-supermartingale if, and only if, $c\leq \frac{m-p}{p}.$
\end{enumerate}
\end{theorem}

\begin{proof}
Recall first the following two facts: (i) $W_{t}^{\ast }\overset{law}{=}%
\sqrt{t}W_{1}^{\ast }$ (by scaling), and (ii) there exists $\eta >0$ such
that $\mathbb{E}\left[ \exp (\eta \left( W_{1}^{\ast }\right) ^{-2})\right]
<+\infty $ (this can be deduced e.g. from \cite[Ch. II, Exercice 3.10]{RY}),
so that the random variable $\left( W_{1}^{\ast }\right) ^{-1}$ has finite
moments of all orders. Note also that the conclusions of both Point 1 and
Point 2 are trivial in the case where $p=m$. In the rest of the proof we
will therefore assume that $p\neq m$.

To prove Point 1, we shall apply Theorem \ref{Th : general} and Proposition %
\ref{P : epsilon} in the following framework: $X_{t}=\left\vert
W_{t}\right\vert ^{m}$ and $\phi \left( x\right) =x^{\frac{p-m}{m}}=x^{\frac{%
p}{m}-1}$. In this case, the martingale $M_{t}=\left\vert W_{t}\right\vert
^{m}-A_{m,t}$ is such that $\left\langle M\right\rangle
_{t}=m^{2}\int_{0}^{t}W_{s}^{2m-2}ds$, $t\geq 0$, and $\Phi \left(
x,z\right) =-\int_{z}^{x}y\phi ^{\prime }\left( y\right) dy=-\left( \frac{p}{%
m}-1\right) \int_{z}^{x}y^{\frac{p}{m}-1}dy=\frac{m-p}{p}\left( x^{\frac{p}{m%
}}-z^{\frac{p}{m}}\right) $. Also, for every $T>\varepsilon >0$%
\begin{eqnarray}
\mathbb{E[}\int_{\varepsilon }^{T}\phi \left( X_{s}^{\ast }\right)
^{2}d\left\langle M\right\rangle _{s}] &=&m^{2}\mathbb{E[}\int_{\varepsilon
}^{T}\left( W_{s}^{\ast }\right) ^{2p-2m}W_{s}^{2m-2}ds]  \notag \\
&\leq &m^{2}\mathbb{E[}\int_{\varepsilon }^{T}\left( W_{s}^{\ast }\right)
^{2p-2}ds]=m^{2}\mathbb{E[}\left( W_{1}^{\ast }\right) ^{2p-2}\mathbb{]}%
\int_{\varepsilon }^{T}s^{\frac{p}{2}-1}ds\text{,}  \label{jjj}
\end{eqnarray}%
so that $\phi $ verifies (\ref{int}) and (\ref{f2}). Relations (\ref{f1})
and (\ref{f4}) are trivially satisfied. To see that (\ref{f3}) holds, use
the relations
\begin{eqnarray*}
\mathbb{E}\left\{ \left\vert \phi \left( X_{\varepsilon }^{\ast }\right)
\left( X_{\varepsilon }-A_{\varepsilon }\right) \right\vert \right\} &=&%
\mathbb{E\{}\left\vert \left( W_{\varepsilon }^{\ast }\right) ^{p-m}\left[
\left\vert W_{\varepsilon }\right\vert ^{m}-A_{m,\varepsilon }\right]
\right\vert \} \\
&=&\mathbb{E}\left\{ \left\vert \left( W_{\varepsilon }^{\ast }\right)
^{p-m}M_{\varepsilon }\right\vert \right\} \leq \mathbb{E}\left\{ \left(
W_{\varepsilon }^{\ast }\right) ^{2p-2m}\right\} ^{1/2}\mathbb{E}\left\{
\left\langle M\right\rangle _{\varepsilon }\right\} ^{1/2} \\
&=&m\mathbb{E}\left\{ W_{1}^{2m-2}\right\} ^{1/2}\mathbb{E}\left\{ \left(
W_{1}^{\ast }\right) ^{2p-2m}\right\} ^{1/2}\varepsilon ^{\frac{p}{2}-\frac{m%
}{2}}\left( \int_{0}^{\varepsilon }s^{m-1}ds\right) ^{1/2} \\
&\rightarrow &0\text{, \ as }\varepsilon \downarrow 0\text{.}
\end{eqnarray*}%
From Point 1 of Proposition \ref{P : epsilon}, we therefore deduce that the
process $Z\left( t\right) $ defined as $Z\left( 0\right) =0$ and, for $t>0$,%
\begin{eqnarray}
Z\left( t\right) &=&\phi \left( \left( W_{t}^{\ast }\right) ^{m}\right)
\left[ \left\vert W_{t}\right\vert ^{m}-A_{m,t}\right] +\alpha \Phi \left(
\left( W_{t}^{\ast }\right) ^{m},0\right)  \label{g} \\
&=&\left( W_{t}^{\ast }\right) ^{p-m}\left[ \left\vert W_{t}\right\vert
^{m}-A_{m,t}\right] +\alpha \frac{m-p}{p}\left( W_{t}^{\ast }\right) ^{p}%
\text{,}  \label{gg}
\end{eqnarray}%
is a $\mathcal{F}_{t}^{W}$-submartingale for every $\alpha \geq 1$. By
writing $c=\alpha \frac{m-p}{p}$ in the previous expression, and by using
the fact that $\frac{m-p}{p}\geq 0$ by assumption, we deduce immediately
that $J_{t}\left( m,c;p\right) $ is a submartingale for every $c\geq \frac{%
m-p}{p}$. Now suppose $c<\frac{m-p}{p}$. One can use formulae (\ref{a}), (%
\ref{g}) and (\ref{gg}) to prove that
\begin{eqnarray*}
J_{t}\left( m,c;p\right) &=&\int_{0}^{t}\phi (X_{s}^{\ast
})dM_{s}+\int_{0}^{t}[-A_{m,s}\phi ^{\prime }\left( (W_{s}^{\ast
})^{m}\right) ]d(W_{s}^{\ast })^{m}+\left( \alpha -1\right) \Phi \left(
\left( W_{t}^{\ast }\right) ^{m},0\right) \\
&=&\int_{0}^{t}(W_{s}^{\ast })^{p-m}dM_{s} \\
&&+\left( \frac{p}{m}-1\right) \int_{0}^{t}[\left( 1-\alpha \right) \left(
W_{s}^{\ast }\right) ^{m}-A_{m,s}](W_{s}^{\ast })^{p-2m}d(W_{s}^{\ast })^{m}%
\text{,}
\end{eqnarray*}%
where $1-\alpha =1-pc/(m-p)>0$. Note that $\int_{0}^{t}(W_{s}^{\ast
})^{p-m}dM_{s}$ is a square-integrable martingale, due to (\ref{jjj}). To
conclude that, in this case, $J_{t}\left( m,c;p\right) $ cannot be a
submartingale (nor a supermartingale), it is sufficient to observe that (for
every $m\geq 1$ and every $\alpha <1$) the paths of the finite variation
process
\begin{equation*}
t\mapsto \int_{0}^{t}[\left( 1-\alpha \right) \left( W_{s}^{\ast }\right)
^{m}-A_{m,s}](W_{s}^{\ast })^{p-2m}d(W_{s}^{\ast })^{m}\text{ }
\end{equation*}%
are neither non-decreasing nor non-increasing, with $\mathbb{P}$-probability
one.

To prove Point 2, one can argue in exactly the same way, and use Point 2 of
Proposition \ref{P : epsilon} to obtain that the process $Z\left( t\right) $
defined as $Z\left( 0\right) =0$ and, for $t>0$,%
\begin{equation*}
Z\left( t\right) =\left( W_{t}^{\ast }\right) ^{p-m}\left[ \left\vert
W_{t}\right\vert ^{m}-A_{m,t}\right] +\alpha \frac{m-p}{p}\left( W_{t}^{\ast
}\right) ^{p}
\end{equation*}%
is a $\mathcal{F}_{t}^{W}$-supermartingale for every $\alpha \geq 1$. By
writing once again $c=\alpha \frac{m-p}{p}$ in the previous expression, and
since $\frac{m-p}{p}\leq 0$, we immediately deduce that $J_{t}\left(
m,c;p\right) $ is a supermartingale for every $c\leq \frac{m-p}{p}$. One can
show that $J_{t}\left( m,c;p\right) $ cannot be a supermartingale, whenever $%
c>\frac{m-p}{p}$, by using arguments analogous to those displayed in the
last part of the proof of Point 1.
\end{proof}

\bigskip

The following result is obtained by specializing Theorem \ref{T : DS} to the
case $m=1$ (via Tanaka's formula).

\begin{corollary}
\label{C : Tanaka}Denote by $\left\{ \ell _{t}:t\geq 0\right\} $ the local
time at zero of the Brownian motion $W$. Then, the process
\begin{eqnarray*}
J_{t}\left( p\right) &=&\left( W_{t}^{\ast }\right) ^{p-1}\left[ \left\vert
W_{t}\right\vert -\ell _{t}\right] +c\left( W_{t}^{\ast }\right) ^{p}\text{,
}t>0\text{,} \\
J_{0}\left( p\right) &=&0\text{,}
\end{eqnarray*}%
is such that: (i) for $p\in (0,1]$, $J_{t}\left( p\right) $ is a $\mathcal{F}%
_{t}^{W}$-submartingale if, and only if, $c\geq 1/p-1$, and (ii) for $p\in
\lbrack 1,+\infty )$, $J_{t}\left( p\right) $ is a $\mathcal{F}_{t}^{W}$%
-supermartingale if, and only if, $c\leq 1/p-1.$
\end{corollary}

\section{Burkholder-Davis-Gundy (BDG) inequalities\label{S : BDG}}

We reproduce an argument taken from \cite[p. 314]{DS}, showing that the
first part of Theorem \ref{T : DS} can be used to obtain a strong version of
the BDG inequalities (see e.g. \cite[Ch. IV, \S 4]{RY}).

Fix $p\in (0,2)$ and define $c=(2-p)/p=2/p-1.$ Since, according to the first
part of Theorem \ref{T : DS}, $Y_{t}=Y_{t}(c,p)$ is a $\mathcal{F}_{t}^{W}$%
-submartingale starting from zero, we deduce that, for every bounded and
strictly positive $\mathcal{F}_{t}^{W}$-stopping time $\tau $, one has $%
\mathbb{E}(Y_{\tau })\geq 0$. In particular, this yields%
\begin{equation}
\mathbb{E}\left( \frac{\tau }{(W_{\tau }^{\ast })^{2-p}}\right) \leq \frac{2%
}{p}\mathbb{E}\left( (W_{\tau }^{\ast })^{p}\right) \text{.}  \label{zio}
\end{equation}%
Formula (\ref{zio}), combined with an appropriate use of H\"{o}lder's
inequality, entails finally that, for $0<p<2$,%
\begin{equation}
\mathbb{E}\left( \tau ^{\frac{p}{2}}\right) \leq \left[ \frac{2}{p}\mathbb{E}%
\left( (W_{\tau }^{\ast })^{p}\right) \right] ^{\frac{p}{2}}\left[ \mathbb{E}%
\left( (W_{\tau }^{\ast })^{p}\right) \right] ^{\frac{2-p}{2}}=\left[ \frac{2%
}{p}\right] ^{\frac{p}{2}}\mathbb{E}\left( (W_{\tau }^{\ast })^{p}\right)
\text{.}  \label{gi}
\end{equation}

Of course, relation (\ref{gi}) extends to general stopping times $\tau $
(not necessarily bounded) by monotone convergence (\textit{via} the
increasing sequence $\left\{ \tau \wedge n:n\geq 1\right\} $).

\bigskip

\textbf{Remark. }Let $\left\{ \mathfrak{A}_{n}:n\geq 0\right\} $ be a
discrete filtration of the reference $\sigma $-field $\mathfrak{A}$, and
consider a $\mathfrak{A}_{n}$-adapted sequence of measurable random elements
$\left\{ f_{n}:n\geq 0\right\} $ with values in a Banach space $\mathbf{B}$.
We assume that $f_{n}$ is a martingale, i.e. that, for every $n$, $\mathbb{E}%
\left[ f_{n}-f_{n-1}\mid \mathfrak{A}_{n-1}\right] =\mathbb{E}\left[
d_{n}\mid \mathfrak{A}_{n-1}\right] =0$, where $d_{n}:=f_{n}-f_{n-1}$. We
note
\begin{equation*}
S_{n}\left( f\right) =\sqrt{\sum_{k=0}^{n}\left\vert d_{k}\right\vert ^{2}}%
\text{ \ \ and \ \ }f_{n}^{\ast }=\sup_{0\leq m\leq n}\left\vert
f_{m}\right\vert ,
\end{equation*}%
and write $S\left( f\right) $ and $f^{\ast }$, respectively, to indicate the
pointwise limits of $S_{n}\left( f\right) $ and $f_{n}^{\ast }$, as $%
n\rightarrow +\infty .$ In \cite{Burkhj}, D.L.\ Burkholder proved that
\begin{equation}
\mathbb{E}\left( S\left( f\right) \right) \leq \sqrt{3}\mathbb{E}\left(
f^{\ast }\right) \text{,}  \label{BI}
\end{equation}%
where $\sqrt{3}$ is the best possible constant, in the sense that for every $%
\eta \in (0,\sqrt{3})$ there exists a Banach space-valued martingale $%
f_{\left( \eta \right) }$ such that $\mathbb{E}\left( S\left( f_{\left( \eta
\right) }\right) \right) >\eta \mathbb{E}\left( f_{\left( \eta \right)
}^{\ast }\right) $. As observed in \cite{DS}, Burkholder's inequality (\ref%
{BI}) should be compared with (\ref{gi}) for $p=1$, which yields
the relation
 $\mathbb{E}\left( \tau ^{1/2}\right) \leq \sqrt{2}\mathbb{E}%
(W_{\tau }^{\ast })$ for every stopping time $\tau $. This shows
that in such a framework, involving uniquely continuous
martingales, the constant $\sqrt{3}$ is no longer optimal.

\section{Balayage \label{S : Bal}}

Keep the assumptions and notation of Section \ref{S : Gen} and Theorem \ref%
{Th : general}, fix $\varepsilon >0$ and consider a finite variation
function $\psi :(0,+\infty )\mapsto \mathbb{R}.$ In this section we focus on
the formula%
\begin{equation}
\psi \left( X_{t}^{\ast }\right) \left( X_{t}-A_{t}\right) -\psi \left(
X_{\varepsilon }^{\ast }\right) \left( X_{\varepsilon }-A_{\varepsilon
}\right) =\int_{\varepsilon }^{t}\psi (X_{s}^{\ast })d\left(
X_{s}-A_{s}\right) +\int_{\varepsilon }^{t}\left( X_{s}^{\ast }-A_{s}\right)
d\psi (X_{s}^{\ast }),  \label{d}
\end{equation}%
where $\varepsilon >0$. Note that by choosing $\psi =\phi $ in (\ref{d}),
where $\phi \in C^{1}$ is monotone, one recovers formula (\ref{uyu}), which
was crucial in the proof Theorem \ref{Th : general}. We shall now show that (%
\ref{d}) can be obtained by means of the \textsl{balayage formulae} proved
in \cite{MY}.\textit{\ }

To see this, let $U=\left\{ U_{t}:t\geq 0\right\} $ be a continuous $%
\mathcal{F}_{t}$-semimartingale issued from zero. For every $t>0$ we define
the random time
\begin{equation}
\sigma \left( t\right) =\sup \left\{ s<t:U_{s}=0\right\} \text{.}
\label{tauti}
\end{equation}%
The following result is a particular case of \cite[Th. 1]{MY}.

\bigskip

\begin{proposition}[Balayage Formula]
Consider a stochastic process $\left\{ K_{t}:t>0\right\} $ such that the
restriction $\left\{ K_{t}:t\geq \varepsilon \right\} $ is locally bounded
and $\mathcal{F}_{t}$-predictable on $\left[ \varepsilon ,+\infty \right) $
for every $\varepsilon >0$. Then, for every fixed $\varepsilon >0$, the
process $K_{\sigma \left( t\right) }$, $t\geq \varepsilon $, is locally
bounded and $\mathcal{F}_{t}$-predictable, and moreover%
\begin{equation}
U_{t}K_{\sigma \left( t\right) }=U_{\varepsilon }K_{\sigma \left(
\varepsilon \right) }+\int_{\varepsilon }^{t}K_{\sigma \left( s\right)
}dU_{s}\text{.}  \label{balabala}
\end{equation}
\end{proposition}

\bigskip

To see how (\ref{d}) can be recovered from (\ref{balabala}), set $%
U_{t}=X_{t}-X_{t}^{\ast }$ and $K_{t}=\psi \left( X_{t}^{\ast }\right) $.
Then, $K_{t}=K_{\sigma \left( t\right) }=\psi (X_{\sigma \left( t\right)
}^{\ast })$ by construction, where $\sigma \left( t\right) $ is defined as
in (\ref{tauti}). As a consequence, (\ref{balabala}) gives
\begin{equation*}
\psi \left( X_{t}^{\ast }\right) \left( X_{t}-X_{t}^{\ast }\right) =\psi
\left( X_{\varepsilon }^{\ast }\right) \left( X_{\varepsilon
}-X_{\varepsilon }^{\ast }\right) +\int_{\varepsilon }^{t}\psi (X_{s}^{\ast
})d\left( X_{s}-X_{s}^{\ast }\right) \text{.}
\end{equation*}%
Finally, a standard integration by parts applied to $\psi \left( X_{t}^{\ast
}\right) \left( X_{t}^{\ast }-A_{t}\right) $ yields
\begin{eqnarray*}
\psi \left( X_{t}^{\ast }\right) \left( X_{t}-A_{t}\right) &=&\psi \left(
X_{t}^{\ast }\right) \left( X_{t}-X_{t}^{\ast }\right) +\psi \left(
X_{t}^{\ast }\right) \left( X_{t}^{\ast }-A_{t}\right) \\
&=&\psi \left( X_{\varepsilon }^{\ast }\right) \left( X_{\varepsilon
}-X_{\varepsilon }^{\ast }\right) +\int_{\varepsilon }^{t}\psi (X_{s}^{\ast
})d\left( X_{s}-X_{s}^{\ast }\right) \\
&&+\psi \left( X_{\varepsilon }^{\ast }\right) \left( X_{\varepsilon }^{\ast
}-A_{\varepsilon }\right) +\int_{\varepsilon }^{t}\psi (X_{s}^{\ast
})d\left( X_{s}^{\ast }-A_{s}\right) \\
&&+\int_{\varepsilon }^{t}\left( X_{s}^{\ast }-A_{s}\right) d\psi \left(
X_{s}^{\ast }\right) \text{,}
\end{eqnarray*}%
which is equivalent to (\ref{d}).

\end{document}